\documentclass[11pt]{amsart}
\usepackage[centering]{geometry} 
\usepackage{hyperref}
\usepackage{graphicx}
\usepackage{amssymb}
\usepackage{epstopdf}
\usepackage{amscd}
\usepackage[english]{babel}

\newtheorem{theorem}{Theorem}[subsection]
\newtheorem{thm}{Main Theorem}[section]
\newtheorem{cor}[thm]{Corollary}
\newtheorem*{mthm}{Main Theorem}
\newtheorem{proposition}[theorem]{Proposition}
\newtheorem{corollary}[theorem]{Corollary}
\newtheorem{claim}[theorem]{Claim}
\newtheorem{lemma}[theorem]{Lemma}
\newtheorem{remark}[theorem]{Remark}
\newtheorem{example}[theorem]{Example}
\newtheorem{definition}[theorem]{Definition}

\newtheorem{summary}[theorem]{Summary}
\newtheorem{fact}[theorem]{Fact}
\newtheorem*{theorem-non}{Theorem}

\DeclareMathOperator{\Spec}{Spec}
\DeclareMathOperator{\Hom}{Hom}
\DeclareMathOperator{\cha}{char}
\DeclareMathOperator{\Sym}{Sym}
\DeclareMathOperator{\norm}{norm}
\DeclareMathOperator{\Proj}{Proj}

\address{
 Courant Institute of Mathematical Sciences,
 New York University, 
 251 Mercer Street, 
 New York, NY 10012, USA }

\email{buonerba@cims.nyu.edu}
\title{Functorial resolution of tame quotient singularities in positive characteristic}
\author{Federico Buonerba}
\begin{document}

\begin{abstract} 
The object of the present is a proof of the existence of functorial resolution of tame quotient singularities for quasi-projective varieties over algebraically closed fields.
\end{abstract}

\maketitle
\tableofcontents
\markboth{FEDERICO BUONERBA}{FUNCTORIAL RESOLUTION OF TAME QUOTIENT SINGULARITIES}

The role of quotient singularities, in the general problem of the resolution of singularities in positive characteristic, has been highlighted in the fundamental work of de Jong \cite{4}, where the general problem has been reduced to the more specific one of understanding singularities created by inseparable morphisms and group actions.

\begin{theorem-non}[\cite{4}, 7.4]
\textit{Let $X$ be a projective variety over an algebraically closed field. There exists a radicial morphism $Y\to X$ and a modification $Z\to Y$ such that $Z$ has at worst quotient singularities.}
\end{theorem-non}

A huge class of group actions enjoy the property of being linearizable, i.e. the action is formally equivalent to that of a group of linear endomorphisms of vector space. In this situation the singularities created by the action are easier to handle, since the existence of formal coordinates along which the action is linear provides a rich amount of information. By way of examples, all tame, \ref{tame}, abelian group actions create singularities that live, \'etale-locally, in the world of toric varieties, see \ref{toric}, and their resolution is in fact a key step for those positive characteristic oriented constructions that lead to a resolution of singularities in characteristic zero, such as those appearing in \cite{3},\cite{4}. As well known, \'etale-locally the resolution problem for tame abelian quotient singularities is completely understood, i.e. there are explicit algorithms, such as \cite{7}, theorems 11 and 11*, and \cite{8}, III.iii.4.bis, that provide us with a resolution. The issue is therefore the local-to-global step (and this is by no means a new story, \cite{10} 3.14.9.), where the toric machinery is not sufficiently strong without additional global assumptions, for example the existence of a toroidal embedding. Consequently it seems convenient to look for a resolution procedure that is functorial under \'etale localization, following the philosophy of \cite{10}, \cite{8}.

Our main result is:
\begin{mthm}[\ref{main}]
Let $k$ be an algebraically closed field, and $X/k$ a quasi-projective variety with tame quotient singularities, with an \'etale cover $\bigsqcup_j X_j\to X$ such that $\max_{x\in X_j}|G_x|$ is finite for every $j$. Then there exists a resolution functor $X\to (M(X),r_X)$ where $M(X)$ is a smooth, quasi projective variety, and $r_X:M(X)\to X$ is a proper, birational, relatively projective morphism, which is an isomorphism over the smooth locus of $X$. The resolution functor commutes with \'etale base change, that is to say for every \'etale morphism $f:Y\to X$ there is a unique isomorphism $\phi_f:f^*M(X)\to M(Y)$.
\end{mthm} 
Now we can discuss the ideas employed. Following the general pattern, \cite{10}, that the construction of a resolution functor is algorithmic in its own nature, we define an invariant $X\to i(X)\in \mathbf N$, such that $i(X)=0$ if and only if $X$ is smooth. Then we proceed by induction on $i$, by performing birational operations - smooth and weighted blow-ups along smooth centers - that eventually decrease $i$. The proof consists of two steps: the first step \ref{step1} constructs a birational modification whose resulting space has only tame cyclic quotient singularities, and the invariant does not increase. Here the full power of the philosophy ``functoriality in the \'etale topology" appears, and most of the logical difficulties, apart from some routine technical issues, are hidden in the inductive step of the process. This inductive step resolves the singularities of a Deligne-Mumford stack with smaller invariant, and the variety we are looking for is its GC quotient, whose existence as an algebraic space is granted by \cite{6}. At this point the inductive hypotheses show up again and force the GC quotient to be a quasi-projective variety, if we started with such.

The second step \ref{step2} is an algorithm that reduces the invariant if we start with tame cyclic quotient singularities and an extra structure of ``global character for the local geometric stabilizers", more precisely an equivariant divisor, on the Vistoli covering stack, along which the stabilizers act with a faithful character. 

The logic how the two steps fit together and provide a resolution is as follows: starting with a variety with tame quotient singularities, the first step produces a variety with tame cyclic quotient singularities and a marked irreducible divisor, \ref{summ}. Properties of this pair are such that we can employ it as input for our second step, whose output is a variety with tame cyclic quotient singularities whose invariant is strictly smaller than the one we started with.
Functoriality in the \'etale topology implies, as explained for example in \cite{10}, that the resolution procedure applies to algebraic spaces, Deligne Mumford-stacks and analytic spaces without any compactness assumption, \ref{coro}, \ref{corol}, \ref{coroll}.
It is worth remarking that, in the course of the proof, we will encounter singularities created by diagonal actions of group schemes of roots of unity, which are not necessarily tame. This must happen if one wants to tackle the problem using weighted modifications/toric geometry in positive characteristic. It turns out that, in this specific circumstance, Mumford's resolution process for toroidal singularities can be carried over by weakening a bit the assumption on the existence of a toroidal embedding, requiring instead the existence of a divisor along which diagonal cyclic stabilizers act with faithful character, see \ref{quasi} for the precise condition.

Unfortunately, the methods developed here seem to completely lose their efficiency if one drops the tameness assumption, in fact the crucial fact we use profusely is that there exist regular parameters along which tame abelian actions are diagonalized, while of course this never happens with non-trivial $p$-group actions in characteristic $p$. We suspect that a resolution functor that deals with all possible quotient singularities simultaneously must proceed via a completely different strategy.

\vspace{0.5cm}

\noindent\textbf{Acknowledgments}: This problem has been suggested by Fedor Bogomolov. His encouragement, and the enlightening conversations we had, have been of great help. This work could never be completed without the support of Michael McQuillan, whose restless and careful explanations permeate and shape the core ideas building our proof. Grateful thanks are extended to Gabriele Di Cerbo, for reading a first draft of this paper and for his editing, that drastically improved our presentation.
 
\section{Generalities}

\subsection{Tame group actions}\label{tame}
Let $k$ be an algebraically closed field and $G$ be a finite group acting by $k$-automorphism on a noetherian, regular, local $k$-algebra $\mathcal O$. The action is said tame if $(|G|,\cha(k))=1$. If $G$ acts on a smooth algebraic variety by $k$-automorphisms, we say that the action is tame if it is generically free, and the stabilizer at any geometric point acts tamely on the corresponding Zariski local ring. From the point of view of singularities we could request more, i.e. the action to be free in codimension one, by way of the celebrated Chevalley-Shephard-Todd theorem:
\begin{fact}\label{pseudo}
 If $G$ is a finite group acting tamely on a smooth algebraic variety, then for any geometric point $x$ with stabilizer $G_x$, the ring of invariants $\mathcal O_x^{G_x}\subset \mathcal O_x$ is a regular local ring if and only if $G_x$ is generated by \textit{pseudoreflections}, i.e every element $g\in G_x$ fixes pointwise a smooth divisor depending on $g$.
\end{fact}

Tame actions enjoy some remarkable properties, with respect to the local algebra they are the simplest possible. We can recollect those properties that we need in the following simple propositions:

\begin{proposition}
Let $G$ be a cyclic group, generated by $g\in G$, acting tamely on the regular local $k$-algebra $\mathcal O$ with maximal ideal $m$. Then there are elements $\zeta_1,...,\zeta_n\in k^*$ and a regular system of parameters $x_1,...,x_n$ generating $m$ such that $x_i^g=\zeta_ix_i$.
\end{proposition}

\begin{proof}
As $g$ acts on the $k$-vector space $m/m^2$, by assumption we can certainly find parameters and eigenvalues satisfying the identities $x_i^g=\zeta_ix_i$ mod $m^2$. We can replace each $x_i$ by $y_i=|G|^{-1}\sum_{k=0}^{k=|G|-1}x_i^{g^k}\zeta_i^{-k}$, and observe that $y_i=x_i$ mod $m^2$, and $y_i$ satisfies the required relation.
\end{proof}

\begin{corollary}\label{fixed}
If $G$ acts tamely on the smooth algebraic variety $X$ then the subvariety $X_G$ of points fixed by $G$ is smooth.
\end{corollary}

\begin{proof}
Let $\mathcal O$ be the local ring at a geometric point $x$ fixed by $G$, then (essentially) as in the previous proof we can average any isomorphism $\hat{\mathcal O}\to k[[x_1,...,x_n]]$ in order to make it equivariant with respect to the induced action of $G$ on the power series ring. Here the action is by linear automorphism, whence the ideal defining $X_G$ is easily seen to be generated by the linear forms $(x_i^g-x_i)_{g\in G}$, and consequently the subvariety $X_G$ is smooth at $x$.
\end{proof}

\begin{proposition}\label{henselian}
Let $\mathcal O$ be a regular local strictly henselian $k$-algebra with maximal ideal $m$, $G$ a finite group acting tamely on it. If $I$ is a principal ideal which is fixed by $G$ then there exists a character $\chi :G\to k^*$ and a generator $f$ of $I$, such that $f^g=\chi (g) f$ for any $g\in G$.
\end{proposition}

\begin{proof}
Let $f$ be a generator of $I$. By assumption, for every $g$ there is $u_g\in \mathcal O^*$ such that $f^g=u_gf$. It follows that $f^{gh}=u_g^hu_hf$, whence the set map $G\to \mathcal O^*$ given by $g\to u_g$ is a 1-cocycle representing a class in $H^1(G,\mathcal O^*)$. Consider the exact sequence $0\to (1+m)\to \mathcal O^*\to k^*\to 0$ of multiplicative groups. By assumption on $\mathcal O$ being strictly henselian, since $|G|$ is coprime to $\cha(k)$ we see that $1+m$ is $|G|$-divisible, whence $H^1(G,1+m)=H^2(G,1+m)=0$. Consequently we have a natural isomorphism $H^1(G,\mathcal O^*)\to H^1(G,k^*)$. Moreover since the $G$-action on the residue field $k$ is trivial, we get $H^1(G,k^*)=\Hom (G,k^*)$, whence, upon replacing $f$ by $uf$, for some invertible $u$, the map $g\to u_g$ is a character of the group $G$, and of course $uf$ is the required generator.
\end{proof}

\begin{proposition}\label{faithful}
Let $q:\mathcal O\to \mathcal O'$ be an injective morphism of rings, equivariant by the action of a cyclic group $C$. Assume $C$ acts with faithful character on $f\in \mathcal O$, then $C$ acts with faithful character on $q(f)$.
\end{proposition}

\begin{proof}
Let $s$ be the minimal integer such that $\chi^sq(f)=q(f)$, then $q(\chi^sf-f)=0$, and since $q$ is injective we deduce $s=|C|$, that is the action is faithful.
\end{proof}

\subsection{Stacks with quotient singularities} 
Let $k$ be an algebraically closed field. An algebraic variety $X/k$ has analytic quotient singularities if it is normal and for any geometric point $x$, there exist a finite group $G_x$ acting on a smooth complete local $k$-algebra $k[[x_1,...,x_n]]$, an inclusion $i_x:\hat{\mathcal O}_{x}\hookrightarrow k[[x_1,...,x_n]]$ such that $i_x(\hat{\mathcal O}_{x})=k[[x_1,...,x_n]]^{G_x}$. By Artin approximation this is equivalent to the fact that $X$ admits an \'etale cover $\bigsqcup V_i\to X$, where $V_i=U_i/G_i$ for smooth varieties $U_i$ and finite groups $G_i$, whose action on $U_i$ is generically free. $X$ has tame quotient singularities if each $G_i$ acts on $U_i$ tamely. Similarly $X$ has algebraic quotient singularities if it admits such an open cover in the Zariski topology.

\begin{definition}
For a closed point $x\in X$, the group $G_x$ is called the geometric stabilizer at $x$.
\end{definition}

A Deligne-Mumford stack $\mathcal X$ has tame quotient singularities if for some \'etale atlas (whence for all) $U\to \mathcal X$ the algebraic variety $U$ has tame quotient singularities.

Before getting into the properties of DM stack with tame quotient singularities, we recall an extremely useful local description of DM stacks:

\begin{fact}[\cite{9} 2.8]
Let $\mathcal X$ be a DM stack over any field, then it admits an \'etale cover by classifying stacks, i.e. stacks of the form $[Y/G]$ with $G$ a finite group.
\end{fact}

\begin{proof}
By \cite{6} 1.1, $\mathcal X$ admits a GC quotient $\mathcal X\to X$. Let $x$ be a geometric point in $X$, with \'etale local ring $\mathcal O_x$, and $U\to \mathcal X$ an \'etale atlas. Consider the following diagram with fibered squares:
$$\begin{CD}
U@<<<\Spec(S)=Y\\
@VVV @VVV \\
\mathcal X@<<<\mathcal X_x\\
@VVV @VVV\\
X@<<<\Spec(\mathcal O_x)
\end{CD}$$ \\ 
where of course $S$ is a strictly henselian local ring and $Y=\Spec (S)\to \mathcal X_x$ is an \'etale atlas.
As the residue field of $S$ is separably closed, $R=Y\times_{\mathcal X_x}Y$ must be a finite disjoint union of copies of $Y$, whence the set $G$ of connected components of $R$ inherits canonically a group structure if we declare that $s$ is the natural projection. Thus $(s,t):R=Y\times G\rightrightarrows Y$is a groupoid with an \'etale morphism $[Y/G]\to \mathcal X$.
\end{proof}

In the situation where $\mathcal X$ is a DM stack with tame quotient singularities there exists a canonical, smooth cover of $\mathcal X$, which we will refer to as \textit{Vistoli covering stack}, constructed as follows:

\begin{fact}[\cite{9} 2.8] 
Let $\mathcal X$ be a DM stack with tame quotient singularities. There exists a smooth DM stack $\mathcal X^V$ with a morphism $\mathcal X^V\to \mathcal X$ which is \'etale in codimension one, satisfying the following universal property: if $\mathcal Y$ is a smooth DM stack and $\mathcal Y\to \mathcal X$ is \'etale in codimension one, then there exists a unique factorization $\mathcal Y\to \mathcal X^V\to \mathcal X$.
\end{fact}

\begin{proof}
Let $W\to \mathcal X$ be an \'etale atlas with tame quotient singularities, say $W=\bigsqcup U_i/G_i$ and let $(s,t): R_{\mathcal X}=W\times_{\mathcal X}W\rightrightarrows W$. Denote by $H_i$ the normal subgroup of $G_i$ generated by pseudo-reflections, let $W^V=\bigsqcup U_i/H_i$ and $R^V$ be the normalization of $W^V\times_{\mathcal X} W^V$ \\ 
$$\begin{CD}
U_i/H_i@<<< s^* (U_i/H_i)  @<<< U_i/H_i\times_{\mathcal X}U_j/H_j @<\text{normalization}<< R^V\\
@VVV @VVV @VVV \\
U_i/G_i @<\text{s}<< R_{\mathcal X} @<<< t^*(U_j/H_j)\\
@VVV @V\text{t}VV @VVV \\
\mathcal X@<<< U_j/G_j@<<< U_j/H_j
\end{CD}$$ \\
$R^V$ is a normal algebraic variety since the diagonal of $\mathcal X$ is representable. The  two projections $(s^V,t^V):R^V\rightrightarrows W^V$ are finite and \'etale in codimension one, since the action of $G_i/H_i$ on $U_i/H_i$ is free in codimension one for every $i$. $W^V$ is smooth by \ref{pseudo}, whence the Zariski-Nagata purity theorem (\cite{5}, X.3.4) implies that the projections $s^V,t^V$ are everywhere \'etale and the groupoid $\mathcal X^V=[W^V/R^V]$ is a smooth DM stack satisfying the required properties.
\end{proof}

Similarly in the situation where $\mathcal X$ is a normal DM stack and $\mathcal D$ is a $\mathbf Q$-Cartier divisor, such that for any geometric point $x\in \mathcal X$, the minimal integer number $n(x)$ such that $n(x)\mathcal D_{|x}$ is Cartier satisfies $(n(x),\cha(k))=1$, there is a canonical cyclic cover of $\mathcal X$ in which $\mathcal D$ becomes everywhere Cartier. We will denote this stack by $\mathcal {X(D)}$, the \textit{Cartification} of $\mathcal D$. The construction is a copy-and-paste extension of that of Gorenstein covering stack:

\begin{proposition}[\cite{11} I.5.3]
There exists a cyclic cover $p_D:\mathcal{X(D)}\to \mathcal X$ such that $p^*\mathcal D$ is Cartier, satisfying the following universal property: if $\mathcal Y$ is a  DM stack and $q:\mathcal Y\to \mathcal X$ is such that $q^*\mathcal D$ is Cartier, then there exists a unique factorization $\mathcal Y\to \mathcal {X(D)}\to \mathcal X$.
\end{proposition}

\begin{proof}
Let $x$ be a geometric point in $\mathcal X$, $V$ an \'etale neighborhood and $V^o$ its smooth locus. Up to shrinking, we can assume that $n\mathcal D$ is generated, over $V$, by a local section $f$ and that $\mathcal D$ is generated, over $V^o$, by a local section $t$. Inside the geometric line bundle $\Spec\ (\Sym\ \mathcal O_{V^o}(D)^{\vee})\to V^o$ consider the subvariety $W^o$ defined by the ideal $T^n-f=0$, where $T$ is a generator of $\Sym\ \mathcal O_{V^o}(D)^{\vee}$ as an $\mathcal O_{V^o}$-algebra. This is a cover of $V^o$ corresponding to the extraction of an $n$-th root of the unity $u\in \mathcal O_{V^o}^*$ satisfying $t^n-uf=0$, which is \'etale since $(n,\cha(k))=1$. This cover extends canonically to a ramified cover $V(\mathcal D)\to V$ by taking the integral closure of $\mathcal O_V$ in the function field of $W^o$. This gives a local description, with respect to an atlas of $\mathcal X$, of an atlas for $\mathcal {X(D)}$, in particular the \'etale local ring of $\mathcal{X(D)}$ at a geometric point $x$ is $\mathcal O_{x}[T]/(T^{n(x)}-f)$. To deduce the groupoid relation, we proceed as in the Vistoli situation, thus we need to check that the natural morphism $(V_i(\mathcal D)\times_{\mathcal X} V_j(\mathcal D))^{\norm}\to V_i(\mathcal D)$ is \'etale. After \'etale localization around a geometric point, this amounts to proving that the natural morphisms $$\mathcal O[T]/(T^n-f)\to \mathcal O[T,S]/(T^n-f,S^m-g)^{\norm}$$
$$\mathcal O[T]/(S^m-g)\to \mathcal O[T,S]/(T^n-f,S^m-g)^{\norm}$$ are \'etale, where $\mathcal O$ is a strictly henselian local ring with a height one prime ideal $I$ such that $(f)=I^n$ and $(g)=I^m$. Let $d=\gcd(n,m)$, $n=n'd$, $m=m'd$ and $w=nm/d$. There is a unit $u\in \mathcal O^*$ such that $f^{m'}=ug^{n'}$, thus we have a relation $T^w=uS^w$, and since $\mathcal O$ is strictly henselian and $(w,\cha(k))=1$ we see that $u^{1/w}\in \mathcal O$. Consequently we deduce a relation $\prod_{\zeta^{w}=1}T-\zeta u^{1/w}S=0$ in $\mathcal O[T,S]/(T^n-f,S^m-g)$. Taking its normalization we get a product of rings $$\prod_{\zeta^{w}=1}\mathcal O[T,S]/(T^n-f,S^m-g, T- \zeta u^{1/w}S)\simeq \prod_{\zeta^{w}=1} \mathcal O[T]/(T^n-f)$$ which implies that the natural morphisms defining the Cartification groupoid $$\mathcal O[T]/(T^n-f)\to \prod_{\zeta^{w}=1} \mathcal O[T]/(T^n-f)$$ $$\mathcal O[S]/(S^m-g)\to \prod_{\zeta^{w}=1} \mathcal O[T]/(T^n-f)$$ given respectively by $T\to \prod_{\zeta^{w}=1} T$ and $S\to \prod_{\zeta^{w}=1} \zeta u^{1/w}T$ are indeed \'etale, thus proving the existence of $\mathcal {X(D)}$ as a DM stack. The universal property follows by the local nature of the construction.
\end{proof}

\begin{remark}
The construction of the Cartification morphism depends heavily on the assumption that the Cartier index of the divisor is coprime to the characteristic of the base field. Without this assumption the local Cartification factors through a non trivial inseparable quotient, and there is no hope to create a DM stack out of it.
\end{remark}

In the sequel we will need a reformulation of the universal property of the Cartification.

\begin{lemma}\label{cartier}
Cartification is stable under pullbacks, that is if $f:\mathcal Y\to \mathcal X$ is any morphism and $\mathcal D$ is $\mathbf Q$-Cartier on $\mathcal X$ then there exists a canonical isomorphism $i_f:\mathcal Y(f^*\mathcal D)\to f^*\mathcal X(\mathcal D)$.
\end{lemma}

\begin{proof}
This is an easy application of the universal property: $\mathcal D$ becomes Cartier on $\mathcal Y(f^*\mathcal D)$ and this affords a morphism $i_f:\mathcal Y(f^*\mathcal D)\to f^*\mathcal X(\mathcal D)$. Since $\mathcal D$ is Cartier on $\mathcal X(\mathcal D)$, so it is after pullback to $f^*\mathcal X(\mathcal D)$, whence by the universal property we deduce an inverse to $i_f$.
\end{proof}

\subsection{Weighted blow-up and characters}\label{weighted}
Let $\mathcal X$ be a smooth DM stack, and $(a_1,...,a_r)$ an $r$-tuple of natural numbers. A blow-up with weights $(a_1,...,a_r)$ is the projectivization $\Proj_{(\oplus_{k\geq 0}I_k/I_{k+1})}\mathcal X\to \mathcal X$, where $I_k$ is a sheaf of ideals such that, \'etale locally, there exist functions $x_1,...,x_r$ forming part of a regular system of parameters and $I_k$ is generated by monomials $x_1^{\alpha_1}...x_r^{\alpha_r}$ with $a_1\alpha_1+...+a_r\alpha_r\geq k$. Then $\oplus_{k\geq 0}I_k/I_{k+1}$ is a sheaf of finitely generated graded algebras.

To give a local description let's assume we are blowing up the origin in the affine space over $k$ with weights $(a_1,...,a_n)$. It is covered by affine open sets $V_i=\Spec(R_i)$ where $R_i=k[ x_1^{\alpha_1}...x_n^{\alpha_n}/x_i^{c_i}$ s.t. $a_ic_i=a_1\alpha_1+...+a_n\alpha_n]$, and the morphism $R_i\to k[y_1,...,y_n]$ given by $x_i\to y_i^{a_i}$, $x_j\to y_jy_i^{a_j}$ induces an isomorphism $\Spec(k[y_1,...,y_n]^{\mu_{a_i}})=\mathbf A^n_k/\mu_{a_i}\to \Spec(R_i)$ where the roots of unity act by way of a generator $g\in \mu_{a_i}$ as follows: $y_i^g=\zeta y_i$, $y_j^g=\zeta^{-a_j}y_j$ with $\zeta$ an actual root of unity in $k^*$. Moreover the exceptional divisor has equation $y_i=0$ in the $i$-th chart. Plainly this local description doesn't make any sense whenever $\cha(k)=p> 0$ since $p$ might divide some $a_i$, so we need more care in defining cyclic group actions:

\begin{definition}\label{diagonal}
Let $\mathcal O$ be a regular local $k$-algebra of dimension $n$, and $l>0$ any natural number. A diagonal action of the cyclic group $C_l$ of order $l$ on $\mathcal O$ with characters $(a_1,..., a_n)$ is a morphism $\mathcal O\to \mathcal O[T]/(T^l-1)$ such that there exists a regular system of parameters $x_1,...,x_n$ with $x_i\to T^{a_i}x_i$. The equalizer of this action will be denoted by $\mathcal O^C$.
\end{definition}

An action of the cyclic group of order $l$ on the smooth DM stack $\mathcal X$ is an action of the group scheme $\Spec(\mathbf Z[T]/(T^l-1))$, i.e. a morphism $\mathcal X\times_{\mathbf Z}\Spec(\mathbf Z[T]/(T^l-1))\to \mathcal X$. It is clear what is means for the action to be diagonal with characters $(a_1,...,a_n)$ around a fixed geometric point.

The definition gives a description of the affine cover of the weighted blow-up in full generality, namely the open sets $V_i$ will be $\Spec(R_i)$ where $R_i$ is the algebra obtained as equalizer of the $\mu_{a_i}$-action on $k[y_1,...,y_n]$ with characters $(-a_1,...,-a_{i-1},1,-a_{i+1},...,a_n)$. Such equalizer is generated, as a $k$-algebra, by monomials, and indeed there is a clear interpretation in terms of toric geometry, \cite{7}.

We will be interested in weighted blow-ups induced by diagonal actions of cyclic groups.

\begin{definition}\label{algebra}
Notation as in \ref{diagonal}, the weighted algebra induced by the characters $(a_1,...,a_n)$ is the sheaf of algebras $A(a_1,...,a_n)=\oplus_{k\geq 0} I_k/I_{k+1}$ on $\Spec(\mathcal O)$, where $I_k=(x_1^{\alpha_1}...x_n^{\alpha_n}|a_1\alpha_1+...+a_n\alpha_n\geq k)$. $C$ acts naturally on $A(a_1,...,a_n)$ and its fixed sub-algebra $B(a_1,...,a_n)=A(a_1,...,a_n)^C$ is canonically a sheaf of algebras on $\Spec(\mathcal O^C)$.
\end{definition}

At which point one might question the dependence of such filtration by ideals $I_k$ on the choices made, which unfortunately turns out to be non-trivial:

\begin{example}
Let $C$ be a cyclic group of order $l$ acting tamely on $k[x,y]$ by way of $x\to \chi x,\ y\to \chi^{-1}y$ for some faithful character $\chi: C\to k^*$. The ring of invariants is $R=k[x^l,y^l,xy]$ whence the action of $\mu_2$ on the plane by way of $x\to y,\ y\to x$ naturally descends to an action on $R$. Observe however that if $l\neq 2$ the weighted algebra induced by the characters $(1,-1)$ of the $C$-action on the plane is not invariant under $\mu_2$, indeed the weighted filtration by ideals is $$I_k=(x^ay^b\ |\ a+(l-1)b\geq k)\subset k[x,y]$$ and this filtration is clearly not invariant under $\mu_2$.
\end{example}

From a global perspective, say on a smooth projective variety $X$ endowed with a tame action of $C$, the example is telling us that local choice of characters (around each fixed component) is by no means sufficient to deduce that the corresponding weighted algebras will patch as we travel around an open cover of $X$. However it turns out that a "global" choice of faithful character is enough to insist that local algebras glue, \ref{glue}.

\begin{definition}
Let $\mathcal O$ be a regular strictly henselian local ring with a diagonal action of a cyclic group $C$. Let $f\in \mathcal O$ be a non-unit, then we say that the action is faithful across the divisor $D(f)=(f=0)$ if, following \ref{diagonal}, there exists a choice of characters for the action of $C$, such that $f\to Tf$. The weighted algebra corresponding to the choice of characters making the action faithful across $D(f)$ will be denoted by $A(D(f))$. The sub-algebra of $C$-invariants will be denoted by $B(D(f))$.
\end{definition}

The following proposition is, broadly speaking, the solution to all our problems, in that it gives a sufficient condition for local weighted algebras to glue to a sheaf of algebras over an algebraic variety:

\begin{proposition}\label{glue}
Let $X$ be a variety with diagonal cyclic quotient singularities, $D$ a reduced, irreducible $\mathbf Q$-Cartier divisor. Assume that for every $x\in D$, $C_x$ is faithful across $D$. Then the local weighted algebras $B(D_x)$ glue to a sheaf of algebras $B$ on $X$.
\end{proposition}

\begin{proof}
First some notation: fix $x\in X$ lying on the support of $D$ and denote by $I_D$ the ideal defining $D$ around $x$. Set $l=|C_x|$. Let $\mathcal O_x^V$ be a strictly henselian, regular local ring such that $\mathcal O_x=(\mathcal O_x^V)^{C_x}$ and $x_1,...,x_n$ a regular system of parameters in $\mathcal O_x^V$ such that $I_D\cdot \mathcal O_x^V=(x_n=0)$, and $C_x$ acts by characters $x_i\to T^{a_i}x_i$ with $a_n=1$. Let $\phi$ be an automorphism of $\mathcal O_x$ fixing the ideal $I_D$. Observe that $I_{D}^l=(x_n^l)$ whence $\phi(x_n^l)=ux_n^l$ for some unit $u$. Consider the elements $y_i=x_n^{l-a_i}x_i\in I_{D}$. We have $$\phi(y_i^l)=\phi(x_n^{l(l-a_i)}x_i^l)=\phi(x_n^{l(l-a_i)})\phi(x_i^l)=u^{l-a_i}x_n^{l(l-a_i)}\phi(x_i^l)$$On the other hand $\phi(y_i)\in I_{D}$ whence $\phi(y_i)=x_n^{b_i}p_i$ for some $b_i>0$ and $p_i\in \mathcal O_x^V$ not a zero divisor in $\mathcal O_x^V/(x_n)$, so then $lb_i=l(l-a_i)$, i.e. $\phi(x_n^{l-a_i}x_i)=x_n^{l-a_i}p_i$. Thus $x_n^{l-a_i}p_i$ is fixed by $C_x$, and the action must be $p_i\to T^{a_i}p_i$ for every $i$. Moreover, the identity 
$$x_1^{\alpha_1}...x_n^{\alpha_n}=(x_1x_n^{l-a_1})^{\alpha_1}...(x_1x_n^{l-a_n})^{\alpha_n}x_n^{\alpha_n-\sum_1^{n-1}\alpha_i(l-a_i)}$$
inside the fraction field $f.f.(\mathcal O_x^V)$ implies more generally that 
$$\phi(x_1^{\alpha_1}...x_n^{\alpha_n})=u^{l^{-1}\sum_{i=1}^n a_i\alpha_i -\sum_{i=1}^{n-1}\alpha_i}p_1^{\alpha_1}...p_{n-1}^{\alpha_1}x_n^{\alpha_n}\\\ (l|\sum_{i=1}^n a_i\alpha_i)$$
which clearly implies that $\phi$ preserves the weighted algebra induced by these characters. Whence local weighted algebras $B(D_x)$ patch to a sheaf of algebras on $X$.
\end{proof}

\subsection{Toric varieties} 
Here we recall basic notation of the theory of toric varieties, following \cite{7}, Chap. 1. The scope of this section is to give a combinatorial interpretation of diagonal cyclic quotient singularities, and re-interpret the observations in \ref{weighted} from this perspective.

Let $V$ be an $n$-dimensional real vector space. The choice of a basis gives a lattice $\mathbf Z^n\subset \mathbf R^n\simeq V$. Let $\sigma\subset V$ be a cone which is generated by finitely many elements in $\mathbf Z^n$. The affine toric variety $X_{\sigma}$ is defined as $\Spec(k[\sigma^*\cap \mathbf Z^n])$, where $\sigma^*$ is the dual cone to $\sigma$, and $k[\sigma^*\cap \mathbf Z^n]$ is the algebra generated by monomials $x_1^{a_1}\dots x_n^{a_n}$ for $(a_1,...,a_n)\in \sigma^*\cap \mathbf Z^n$. It is easy to see that toric open subsets of $X_{\sigma}$ correspond similarly to the faces of $\sigma$, whence whenever $\tau\subset \mathbf R^n$ is a fan (i.e. a polyhedron which is obtained as union of finitely many, finitely generated cones intersecting along common faces) we have a toric variety $X_{\tau}$ obtained by glueing the affine toric varieties corresponding to the cones defining $\tau$.

A fundamental construction admitting an interpretation in terms of the toric cone is that of blow-up: given a toric cone $\sigma =<v_1,...,v_m>$ and a vector $v\in \sigma$ which is not in $\sum_{i=1}^m \mathbf N v_i$, we can consider the fan $\sigma_v$ obtained as the union of the cones $\sigma_i=< v_1,...,v_{i-1},v,v_{i+1},...,v_n >$. Then there exists a sheaf of ideals $I_v$ in $X_{\sigma}$ such that the normalization of $Bl_{I_v}X_{\sigma}$ is isomorphic to $X_{\sigma_v}$, and the toric morphism induced by the inclusion $\sigma_v\to \sigma$ corresponds to the natural projection $(Bl_{I_v}X_{\sigma})^{\norm}\to X_{\sigma}$ under this isomorphism. Now we turn our attention to the class of cyclic quotient singularities by observing that a formal germ of cyclic quotient singularity is always toric, by a trivial extension of \cite{7}, pp. 16-18:

\begin{proposition}\label{toric}
Let $e_1,...,e_n$ be a basis of a given real $n$-dimensional vector space $V$ and let $a_1,...,a_{n-1},l$ be natural numbers, $l\neq 0$. Then the cone $\sigma$ generated by $e_1,...,e_{n-1},le_n-(a_1e_1+...+a_{n-1}e_{n-1})$ corresponds to the quotient of $\mathbf A^n_k$ by the action of the cyclic group of order $l$ with characters $a_1,...,a_{n-1},1$.
\end{proposition}

\begin{proof}
The basis $e_1,...,e_n$ determines a free group $\mathbf Z^n\subset V$. Let $N\subset \mathbf Z^n$ be the subgroup of index $l$ generated by $e_1,...,e_{n-1},le_n-(a_1e_1+...+a_{n-1}e_{n-1})$, and let $M$ be its dual lattice in $V^*$. Observe that $\sigma^*$ is generated by $$e_i^*+a_il^{-1}e_n^*\   \quad (i\leq n-1), \quad l^{-1}e_n$$ which is basis of $M$, whence $\Spec(k[\sigma^*\cap M])=\mathbf A^n_k$ with coordinates $x_1,...,x_n$ corresponding to this choice of generators for $\sigma^*$. Also the choice of a generator $\zeta\in \mu_l$ gives an isomorphism $\mathbf Z^n/N\simeq \mu_l$ obtained by sending $e_n\to \zeta$. Consider now the $\mu_l$ action on $\Spec(k[\sigma^*\cap M])=\mathbf A^n_k$ with characters $(a_1,..,a_{n-1},1)$. On monomials the action is given by $$\mathbf x^m\to T^{lm(e_n)}\mathbf x^m\ \quad    (m\in \sigma^*\cap M)$$ whence the fixed subalgebra is generated by monomials $\mathbf x^m$ such that $m(e_n)\in \mathbf Z$, which is evidently $k[\sigma^*\cap \mathbf Z^n]$.
\end{proof}

The next observation expands the toroidal view on diagonal cyclic quotient singularities \ref{toric}, by giving a toroidal interpretation of weighted blow-ups induced by characters of diagonal cyclic actions:

\begin{proposition}\label{decomposition}
Let $C$ be a cyclic group of order $l$ acting on $\mathbf A^n_k$ diagonally with characters $a_1,...,a_{n-1},1$, and let $X_{\sigma}$ be the quotient variety. Let $A=A(a_1,...,a_{n-1},1)$ and $B=B(a_1,...,a_{n-1},1)$, cf \ref{algebra}. Then $\Proj_B(X_{\sigma})$ has only diagonal cyclic quotient singularities of orders $a_1,...,a_{n-1}$. Moreover such modification coincides with the toric blow-up obtained by way of the decomposition of $\sigma$, as defined in \ref{toric}, by adding the vector $e_n$.
\end{proposition}

\begin{proof}
By the local description given at the beginning of \ref{weighted}, $\Proj_A(\mathbf A^n_k)$ has diagonal cyclic quotient singularities of orders $a_1,...,a_{n-1}$. The action of $C$ lifts to $\Proj_A(\mathbf A^n_k)$, and it acts by pseudoreflections across the exceptional divisor of $\Proj_A(\mathbf A^n_k)\to \mathbf A^n_k$. It follows that $\Proj_A(\mathbf A^n_k)/C=\Proj_{B}(X_{\sigma})$ has diagonal cyclic quotient singularities of orders $a_1,...,a_{n-1}$. Inspection of the local structure of $\Proj_{B}(X_{\sigma})$ around the exceptional divisors reveals immediately that the said toric blow-up gives the same output.
\end{proof}

\section{Resolution}

The aim of this section is to prove the main theorem.

\begin{thm}\label{main}
Let $k$ be an algebraically closed field, and $X/k$ a quasi-projective variety with tame quotient singularities, with an \'etale cover $\bigsqcup_j X_j\to X$ such that $\max_{x\in X_j}|G_x|$ is finite for every $j$. Then there exists a resolution functor $X\to (M(X),r_X)$ where $M(X)$ is a smooth, quasi projective variety, and $r_X:M(X)\to X$ is a proper, birational, relatively projective morphism which is an isomorphism over the smooth locus of $X$. The resolution functor commutes with \'etale base change, that is to say for every \'etale morphism $f:Y\to X$ there is a unique isomorphism $\phi_f:f^*M(X)\to M(Y)$.
\end{thm}

Before getting into the proof, let's observe a few immediate corollaries in the spirit of \cite{10}, 3.42, 3.43, 3.44:

\begin{cor}\label{coro}
Let $X$ be an analytic space with tame quotient singularities. Then there exists a relatively projective, proper and birational morphism $X'\to X$ with $X'$ smooth.
\end{cor}

\begin{proof}
We can find an \'etale cover of $X$, consisting of at most countably many analytic spaces satisfying the condition that each of them has bounded geometric stabilizers.
\end{proof}

\begin{cor}\label{corol}
Let $\mathcal X$ be a DM stack with tame quotient singularities. Then there exists a representable, proper, birational, relatively projective morphism $\mathcal X'\to \mathcal X$ with $\mathcal X'$ smooth.
\end{cor}

\begin{proof}
We can represent $\mathcal X$ as an \'etale groupoid $(s,t):R\rightrightarrows X$. Functoriality in the \'etale topology gives unique isomorphisms $\phi_s:M(R)\to s^*M(X)$ and $\phi_t: M(R)\to t^*M(X)$, and the induced morphisms $(s\phi_s,t\phi_t):M(R)\rightrightarrows M(X)$ inherit canonically the structure of a groupoid, whose projection onto $\mathcal X$ is clearly representable.
\end{proof}

\begin{cor}\label{coroll}
The resolution functor commutes with smooth base change.
\end{cor}

\begin{proof}
This is \cite{10} 3.9.2.
\end{proof}

The rest of the manuscript will be devoted to the proof of \ref{main}.

\vspace{0.5cm}

\subsection{Step 1: Reduction to tame cyclic quotient singularities.}\label{step1}

\noindent We can assume that $X=X_j$ for some index $j$. The invariant on which the inductive argument builds upon is $i(X)=\max_{x\in X}|G_x|$. By \ref{fixed} the subvariety $Z$ consisting of points with geometric stabilizer of maximal cardinality is smooth, so let $p_Z:Y=Bl_ZX\to X$ be its blow-up with exceptional divisor $E$. If $i(Y)<i(X)$ we are done, otherwise necessarily $i(Y)=i(X)$ and $E$ is $\mathbf Q$-Cartier but not Cartier, whence there is a non trivial Cartification cover $q_E:Y(E)\to Y$.

\begin{claim} 
$i(Y(E))<i(Y)$.
\end{claim}

\begin{proof}
Consider the natural projection $q_V:Y^V\to Y(E)$ and let $y\in Y^V$ be a geometric point with $|G_y|=i(Y)$. The projection induces a morphism of strictly henselian local rings at $y$, $\mathcal O_{Y(E),q_V(y)}\to \mathcal O_{Y^V,y}$. The action of $G_y$ on $\mathcal O_{Y^V,y}$ leaves $E_y$ invariant, so by \ref{henselian} it is given by a non-trivial character $\chi_{E,y}:G_y\to k^*$. We see immediately that $\mathcal O_{Y,q_Eq_V(y)}$ is the subring of $\mathcal O_{Y(E),q_V(y)}$ of elements fixed by the cyclic action of $Im(\chi_{E,y})$. Thus $G_{q_V(y)}=ker(\chi_{E,y})$ and the claim follows.
\end{proof}

Now we can apply our inductive hypothesis, in the formulation \ref{corol}, and obtain a representable, proper, birational, relatively projective morphism $r_Y:Y_1\to Y(E)$ with $Y_1$ smooth DM stack, along with a GC quotient morphism $m:Y_1\to X_1$, \cite{6} 1.1. Since $Y$ is quasi-projective and $r_Y$ is relatively projective, $X_1$ is quasi-projective. The total pullback of $r_Y^*q_E^*E$ naturally descends to a $\mathbf Q$-Cartier divisor $E_1$ on $X_1$, i.e. $m^*E_1=r_Y^*q_E^*E$. The pair $(X_1,E_1)$ enjoys particularly good properties, that we can summarize as follows:

\begin{claim}
$Y_1=X_1(E_1)$.
\end{claim}

\begin{proof}
In our setting the GC quotient morphism $m:Y_1\to X_1$ is locally a cyclic cover, say given by the tame cyclic action of $C_y$ in a neighborhood of a geometric point $y\in Y_1$. The assertion we need to prove is that $C_y$ acts faithfully across $r_Y^*q_E^*E$. Observe that $C_y$ is faithful on $q_E^*E\subset Y(E)$ in a neighborhood of $r_Y(y)$, and since $Y_1\to Y(E)$ is birational, whence injective on local rings, we conclude by \ref{faithful} that $C_y$ is faithful on $r_Y^*q_E^*E$ as well. 
\end{proof}

It is pretty straightforward to control the behaviour of our construction under \'etale base change:

\begin{proposition}
Let $f:X'\to X$ be an \'etale morphism, then the square
$$\begin{CD}
X'_1 @>>>X_1\\
@VVV @VVV\\
X' @>>>X
\end{CD}$$
is fibered.
\end{proposition}

\begin{proof}
Consider the diagram
$$\begin{CD}
X'@<<<  Y'=Bl_{f^*Z}X' @<<< Y(f^*E)'@<\text{Resolution}<< Y_1'@>\text{GC quotient}>> X_1'\\
@V\text{\'etale}VV @VVV @VVV @VVV @VVV \\
X@<<<  Y=Bl_Z(X) @<<< Y(E)@<\text{Resolution}<< Y_1@>\text{GC quotient}>> X_1\\
\end{CD}$$ \\
The two leftmost squares are fibered, by \ref{cartier}. The third square from the left is also fibered, by our inductive hypothesis, thus $Y_1'= Y_1\times_X X'$, and we conclude since the GC quotient functor commutes with flat base change - as it follows by definition, \cite{6} 1.8.
\end{proof}

We conclude the first step of the proof by summarizing what we got so far.

\begin{summary}\label{summ}: Given a quasi-projective variety $X$ satisfying the assumptions of \ref{main}, there exists a functor $X\to (X_1,E_1=E_1(X),p_X)$ where $X_1$ is a quasi projective variety with tame cyclic quotient singularities with $i(X_1)\leq i(X)$, $E_1$ is a $\mathbf Q$-Cartier divisor such that $X_1^V\to X_1(E_1)$ is an isomorphism. $p_X:X_1\to X$ is a proper, relatively projective, birational morphism. The functor commutes with \'etale base change, that is to say for every \'etale morphism $f:Y\to X$ there is a unique isomorphism $\psi_f:f^*X_1\to Y_1$ such that $\psi_f^*p_Y^*E_1(Y)=p_X^*E_1(X)$. Here is a diagram summarizing our construction:
$$\begin{CD}
q_V^*q_E^*E\subset Y_V\\
@V\text{Vistoli}VV\\
q_E^*E\subset Y(E) @<\text{resolution}<< m^*E_1\subset Y_1\\
@V\text{Cartification}VV @V\text{GC quotient=Cartification}VV\\
E\subset Bl_ZX=Y@<<< E_1\subset X_1\\
@VVV\\
Z\subset X
\end{CD}$$
\end{summary}

\vspace{0.3cm}

\subsection{Step 2: Reduction of the invariant for tame cyclic quotient singularities.}\label{step2}
The second step of the proof works out an algorithmic procedure to reduce the invariant $i(X)$, if we are given a pair $(X,E)$ where $X$ is a quasi-projective variety with tame cyclic quotient singularities, $E$ is a reduced divisor whose connected components are irreducible, and the canonical projection $X^V\to X(E)$ is an isomorphism - meaning that the local cyclic groups defining the Vistoli cover $X^V\to X$ act with a faithful character on the pullback of $E$ in $X^V$. The key observation we need to construct a resolution is \ref{glue}, that for such a pair $(X,E)$ there exists a choice of characters for the cyclic action whose corresponding local weighted algebras form a globally well defined sheaf of algebras. Blowing it up, \ref{decomposition}, we obtain a quasi-projective variety with cyclic quotient singularities whose stabilizers are smaller, but these need not be tame anymore. However the cyclic quotients need not be tame for weighted blow-ups to make sense, whence further applications of \ref{decomposition}, in the form \ref{quasi}, can be used to remove the non-tame stabilizers that might appear after the first blow-up. More precisely let $Z$ be the smooth subvariety of points with maximal geometric stabilizers. By assumption, the stabilizers are faithful across $E$, whence by \ref{glue} there is a sheaf of algebras $B_Z$ induced by the faithful actions of the generic stabilizers at $Z$ along $E$. Let $p_w:X^w:=\Proj_{B_Z}X\to X$ and consider the pair $(X^w,E^w:=p_w^*E)$. We have two cases:
\begin{itemize}
\item[(i)] $X^w$ has tame cyclic quotient singularities. In this case $i(X^w)<i(X)$ and the main theorem follows.
\item[(ii)] $X^w$ has non-tame, diagonal cyclic quotient singularities. In this case we can eliminate the non-tame cyclic stabilizers by way of the following.
\end{itemize}

\begin{lemma}\label{quasi}
Let $(X,D)$ be a pair where $X$ is a quasi-projective variety with diagonal cyclic quotient singularities and $D=\sum_{i=1}^t D_i$ is a reduced divisor. Assume that, if $x\in X$ is such that $C_x$ is not tame, then $x\in D$ and $C_x$ is faithful across at least one irreducible component of $D$. Then there exists a resolution functor $(X,D)\to (M(X,D),u_X)$ where $M(X,D)$ is a quasi projective variety with tame cyclic quotient singularities such that $i(M(X,D))\leq \max_{x\in X}|C_x|$, $u_X:M(X,D)\to X$ is a proper, birational, relatively projective morphism which is an isomorphism over the smooth locus of $X$. The  functor commutes with \'etale base change, that is to say for every \'etale morphism $f:(Y,D_Y)\to (X,D)$ there is a unique isomorphism $\phi_f:f^*M(X,D)\to M(Y,D_Y)$.
\end{lemma}

\begin{proof}
Let $j(X)=\max\{|C_x|$, $x\in X$ s.t. $C_x$ is not tame$\}$. Clearly the subvariety $Z$ of points with maximal non-tame geometric stabilizer is a smooth subvariety of $D$. We will proceed by induction on $j(X)$. Let $Y$ be a connected component of $Z$, and denote by $h(Y)$ the maximal index such that $Y\subset D_{h(Y)}$ and the stabilizer $C_Y$ is faithful across $D_{h(Y)}$. By \ref{glue} for every such $Y$ there is a sheaf of algebras $B_Y$ supported on $Y$ and induced by the faithful action of the geometric stabilizers across $D_{h(Y)}$. Let $B_Z=\sum_Y B_Y$, and denote by $p_1:X_1=\Proj_{B_Z}X\to X$ with exceptional divisor $D_{t+1}$ and set $D_1=p_1^*D=\sum D_i^1+D_{t+1}$ where of course $D_i^1$ is the strict transform of $D_i$ under $p_1$. By \ref{decomposition} $j(X_1)<j(X)$, and we claim that the pair $(X_1,D_1)$ still satisfies the assumptions of \ref{quasi}. Indeed let $Y_1$ be a connected component of $Z_1$, if $Y_1$ is not contained in $D_{t+1}$ then $Y_1$ is the proper transform of some component with non-tame stabilizer $Y\subset X$, and faithfulness across $D_{h(Y)}^1$ is granted by \ref{faithful}. Next if $Y_1\subset D_{t+1}$, then the local description of a weighted blow-up, as given at the beginning of \ref{weighted}, forces the geometric stabilizers along $Y_1$ to be faithful across $D_{t+1}$. By induction on $j(X)$ we conclude the existence of $(M(X,D),u_X)$. Functoriality in the \'etale topology is a trivial consequence of our iterative construction.
\end{proof}

We are ready to conclude the proof of the main theorem: starting with a quasi-projective variety with tame quotient singularities $X$, \ref{step1} gives a pair $(X_1,E_1)$ as in \ref{summ}. Running \ref{step2} with input $(X_1,E_1)$ we get a quasi-projective variety with tame cyclic quotient singularities $M(X_1^w,E_1^w)$ - where of course $M(X_1^w,E_1^w)=X^w_1$ in case i) above - such that $i(M(X_1^w,E_1^w))<i(X)$, along with a proper, relatively projective, birational morphism $M(X_1^w,E_1^w)\to X$. \qed

\end{document}